\documentclass{amsart}
\usepackage{amssymb}
\usepackage{color}
\usepackage{hyperref}
\usepackage{tikz-cd}
\usepackage{rotating}
\newcommand*{\isoarrow}[1]{\arrow[#1,"\rotatebox{90}{\(\simeq\)}"
	]}
\DeclareFontFamily{U}{wncy}{}
\DeclareFontShape{U}{wncy}{m}{n}{<->wncyr10}{}
\DeclareSymbolFont{mcy}{U}{wncy}{m}{n}
\DeclareMathSymbol{\Sha}{\mathord}{mcy}{"58}

\newtheorem{theorem}{Theorem}[section]
\newtheorem{lemma}[theorem]{Lemma}
\newtheorem{proposition}[theorem]{Proposition}

\newtheorem{corollary}[theorem]{Corollary}
\theoremstyle{definition}
\newtheorem{definition}[theorem]{Definition}
\newtheorem{example}[theorem]{Example}

\newtheorem{remark}[theorem]{Remark}
\numberwithin{equation}{section}
\usepackage[all]{xy}
\DeclareMathOperator{\Po}{Po}
\DeclareMathOperator{\Ost}{Ost}
\DeclareMathOperator{\Pic}{Pic}
\DeclareMathOperator{\Cl}{Cl}

\DeclareMathOperator{\Gal}{Gal}

\DeclareMathOperator{\Ker}{Ker}
\DeclareMathOperator{\Coker}{Coker}
\DeclareMathOperator{\Image}{Image}
\DeclareMathOperator{\Hom}{Hom}
\DeclareMathOperator{\cts}{cts}
\DeclareMathOperator{\Inf}{Inf}
\DeclareMathOperator{\Res}{Res}

\DeclareMathOperator{\trans}{trans}
\DeclareMathOperator{\Trans}{Trans}

\begin{document}
	
	\title{The Ostrowski quotient of an elliptic curve}

	\author{Abbas Maarefparvar}
	\address{School of Mathematics, Institute for Research in Fundamental Sciences (IPM), P.O. Box: 19395-5746, Tehran, Iran.}
	\curraddr{}
	\email{a.marefparvar@ipm.ir}
	\thanks{This research was supported by a grant from IPM}
	\date{}
	\dedicatory{}
	\commby{}
	

	\begin{abstract}
		For $K/F$ a finite Galois extension of number fields, the relative P\'olya group $\Po(K/F)$ is the subgroup of the ideal class group of $K$ generated by all the strongly ambiguous ideal classes in $K/F$. The notion of Ostrowski  quotient $\Ost(K/F)$, as the cokernel of the capitulation map into $\Po(K/F)$, has been recently introduced in \cite{SRM}.
		 In this paper, using some results of Gonz\'alez-Avil\'es \cite{Aviles}, we find a new approach to define $\Po(K/F)$ and $\Ost(K/F)$ which is the main motivation for us to investigate analogous notions in the elliptic curve setting.  For $E$ an elliptic curve defined over $F$, we define the Ostrowski quotient $\Ost(E,K/F)$ and the coarse Ostrowski quotient $\Ost_c(E,K/F)$ of $E$ relative to $K/F$, for which in the latter group we do not take into account primes of bad reduction. Our main result is a non-trivial structure theorem for the group $\Ost_c(E,K/F)$ and we analyze this theorem, in some details, for the class of curves $E$ over quadratic extensions $K/F$.
	\end{abstract}

	\maketitle
	
	\vspace{.2cm} {\noindent \bf{Keywords:}}~ Tate-Shafarevich group,  transgression map, localization map, Ostrowski quotient, coarse Ostrowski quotient.

	\vspace{0.3cm}	
	
	{\noindent \bf{Notation.}}~  The following notations will be used throughout this article:

For a number field $F$, the notations $I(F)$, $P(F)$, $\Cl(F)$, $h_F$, $\mathcal{O}_F$, $U_F$ and $\mathbb{P}_F$ denote the group of fractional ideals, group of principal fractional ideals, ideal class group, class number, ring of integers, group of units and set of all prime ideals of $F$, respectively. For a finite extension $K/F$ of number fields, we use  $N_{K/F}$ to denote both the ideal norm and the element norm map from $K$ to $F$. Also ${\epsilon}_{K/F}:\Cl(F) \rightarrow \Cl(K)$ denotes the capitulation map induced by the extension homomorphism from $I(F)$ to $I(K)$. For a prime ideal $\mathfrak{p} \in \mathbb{P}_F$, we denote by $e_{\mathfrak{p}(K/F)}$   the ramification index of $\mathfrak{p}$ in $K/F$.

\vspace{.2cm}	
	
	\section{Introduction}  \label{section, Introduction}
	
Let $K$ be an algebraic number field. For a prime power $q$, the \textit{Ostrowski ideal} $\Pi_{q}(K)$ of $K$ is defined as follows \cite{Ostrowski}:
\begin{equation} \label{equation, Ostrowski ideals}
	\Pi_{q}(K) :=\prod_{\substack{\mathfrak{p}\in\mathbb{P}_K\\ N_{K/\mathbb{Q}}(\mathfrak{p})=q}} \mathfrak{p}.
\end{equation}
By convention, if $K$ has no ideal with norm $q$, we set $\Pi_{q}(K)=\mathcal{O}_K$. The number field $K$ is called a \textit{P\'olya field} if all the Ostrowski ideals $\Pi_{q}(K)$ for arbitrary prime powers $q$ are principal \cite{Zantema}. The notion of \textit{P\'olya-Ostrowski group} or \textit{P\'olya group} was introduced in \cite[Chapter II, $\S$3]{Cahen-Chabert's book}.
\begin{definition} \label{definition, Polya group}
	The P\'olya group of a number field $K$ is the subgroup $\Po(K)$ of
	the class group $\Cl(K)$ generated by the classes of all the Ostrowski ideals $\Pi_q(K)$.
\end{definition} 

Hence $K$ is a P\'olya field if and only if $\Po(K)$ is trivial. Obviously every number field with class number one is a P\'olya field, but not conversely. For instance, Zantema proved that every cyclotomic field is a P\'olya field \cite[Proposition 2.6]{Zantema}. The reader is referred to \cite{Cahen-Chabert's book, ChabertI, MR1, MR2, Zantema} for some results on P\'olya fields and P\'olya groups. Recently the notion of P\'olya group has been relativized  independently in \cite{MaarefparvarThesis, ChabertI}.
	
	\begin{definition} 
		For $K/F$ a finite extension of number fields, the \textit{relative P\'olya group} of $K$ over $F$, is the subgroup of $\Cl(K)$ generated by the classes of the \textit{relative Ostrowski ideals}
		\begin{align} \label{equation, relative Ostrowski ideals}
			\Pi_{\mathfrak{p}^f}(K/F):=\prod_{\substack{\mathfrak{P}\in \mathbb{P}_K \\ N_{K/F}(\mathfrak{P})=\mathfrak{p}^f}} \mathfrak{P},
		\end{align}
		where $\mathfrak{p}$ is a prime ideal of $F$, $f$ is a positive integer and by convention, if $K$ has no ideal with relative norm $\mathfrak{p}^f$ (over $F$), then $\Pi_{\mathfrak{p}^f}(K/F)=\mathcal{O}_K$.
		We denote the relative P\'olya group of $K$ over $F$ by $\Po(K/F)$. In particular, $\Po(K/\mathbb{Q})=\Po(K)$ and $\Po(K/K)=\Cl(K)$.
	\end{definition}

	If $K/F$ is a Galois extension with Galois group $G$, the ideals $\Pi_{\mathfrak{P}^f}(K/F)$ freely generate the ambiguous ideals $I(K)^G$ \cite[Proof of Proposition 2.2]{Brumer-Rosen}. Hence in this case,  $\Po(K/F)$ coincides with the group of strongly ambiguous ideal classes of $K/F$. Moreover, using some results of Brumer and Rosen \cite[$\S$2]{Brumer-Rosen}, one may find the following exact sequence which is also obtained by Zantema in the absolute case \cite[$\S$3]{Zantema}.

	\begin{proposition} \cite{ChabertI,MR2} \label{theorem, generalization of the Zantema's exact sequence}
		For $K/F$ a finite Galois extension of number fields with Galois group $G$, the following sequence is exact:
		{\small
			\begin{equation*} \label{equation, BRZ exact sequence}
				\tag{BRZ} 
				0 \rightarrow \Ker({\epsilon}_{K/F}) \rightarrow H^1(G,U_K) \rightarrow\bigoplus_{\mathfrak{p} \in \mathbb{P}_F} \frac{\mathbb{Z}}{e_{\mathfrak{p}(K/F)}\mathbb{Z}} \rightarrow \frac{\Po(K/F)}{{\epsilon}_{K/F}(\Cl(F))} \rightarrow 0.
		\end{equation*}}
	\end{proposition}
	
	\begin{remark}
	Note that	``BRZ'' stands for Brumer-Rosen \cite{Brumer-Rosen} and Zantema \cite{Zantema} whose cohomological results are the main tools to prove the above proposition, see \cite[Remark 1.4]{SRM}.
	\end{remark}

	As a modification of the relative P\'olya group, the notion of ``Ostrowski quotient'' has been recently introduced in \cite{SRM}.

	\begin{definition}  \label{definition, Ostrowski group}
		For $K/F$ a finite extension of number fields, the \textit{Ostrowski quotient} $\Ost(K/F)$ is defined as  
		\begin{equation} \label{equation, Ostrowski group general case}
			\Ost(K/F) := \frac{\Po(K/F)}{\Po(K/F) \cap \epsilon_{K/F}(Cl(F))}.
		\end{equation} 
		In particular, $\Ost(K/\mathbb{Q})=\Po(K/\mathbb{Q})=\Po(K)$ and $\Ost(K/K)=0$. The extension $K/F$ is called \textit{Ostrowski} (or $K$ is called $F$-Ostrowski) whenever $\Ost(K/F)$ is trivial.
	\end{definition}

	\begin{remark} \label{remark, if L/K is Galois then epsilon of Cl(K) is contained in relative Polya group}
		If $K/F$ is Galois, then $\epsilon_{K/F}(\Cl(F)) \subseteq \Po(K/F)$ \cite[$\S$ 2]{MR2}. Hence
		\begin{equation} \label{equation, for K/F Galois extension, Ost(K/F) equals to Po(K/F) modulo image of epsilon}
			\Ost(K/F)=\frac{\Po(K/F)}{\epsilon_{K/F}(\Cl(F))}.
		\end{equation}
	\end{remark}

	Some applications of \eqref{equation, BRZ exact sequence}, especially generalizations of some results for P\'olya fields, have been found in  \cite[$\S2$]{MR2}. Moreover, using \eqref{equation, BRZ exact sequence} one may find simple proofs for some classical well known results in the literature for the capitulation problem \cite[$\S$2.2.1]{SRM}. For more applications of \eqref{equation, BRZ exact sequence}, see \cite{ChabertI,MR2,SRM}.

\section{Localization and transgression maps} \label{section, Localization and transgression maps}

From now on, we fix the following notations:

\begin{itemize}
	\item[•]
	$K/F$: a finite Galois extension of number fields with Galois group $G$;
	
	\vspace*{0.15cm}
	\item[•]
	$G_F:=Gal(\overline{F}/F)$, where $\overline{F}$ is the algebraic closure of $F$;
	\vspace*{0.15cm}
	\item[•]
	$M_F$: a complete set of places of $F$;
	\vspace*{0.15cm}
	
		\item[•]
	$M_F^{\infty}$: the set of all infinite places of $F$;
		\vspace*{0.15cm}
	\item[•]
	$M_F^0$: the set of all finite places of $F$;
	\vspace*{0.15cm}
	\vspace*{0.15cm}
	\item[•]
	$F_v$: the completion of $F$ at a place $v \in M_F$;
	\vspace*{0.15cm}
	\item[•]
	$w$: a fixed place of $K$ lying above $v \in M_F$;
	\vspace*{0.15cm}
	\item[•]
	$U_{K_{w}}$: the group of local units in $K_{w}$;
	\vspace*{0.15cm}
	\item[•]
	$G_{w}:=Gal(K_{w}/F_v)$ for a place $v \in M_F$;
	\vspace*{0.15cm}
	\item[•]
	$G_{F_v}:=Gal(\overline{F_v}/F_v)$, where $\overline{F_v}$ is the algebraic closure of $F_v$.
\end{itemize}
\vspace*{0.15cm}

Let $S$ be a finite set of primes of $F$ with $M_F^{\infty} \subseteq S$. 
In \cite{Aviles}, Gonz\'alez-Avil\'es found some interesting results on the kernel and cokernel of the \textit{$S$-capitulation map}. Assuming $S=M_F^{\infty}$, we summarize a few results of Gonz\'alez-Avil\'es which we need for the rest of the paper.

\begin{proposition} \cite[$\S$2]{Aviles} \label{proposition, Aviles results}
	For $K/F$ a finite Galois extension of number fields with Galois group $G$, the following sequences are exact:
\begin{align}
	& 0 \rightarrow P(K)^G \rightarrow I(K)^G \rightarrow \Cl(K)_{\trans}^G \rightarrow 0,
	\\
	&0 \rightarrow \Cl(K)_{\trans}^G \rightarrow \Cl(K)^G  \rightarrow  H^1(G, P(K)) \rightarrow 0, 
	\\
	& 0  \rightarrow H^1(G,P(K)) \rightarrow H^2(G,U_K) \rightarrow H^2(G,K^{\times}),
\end{align}
		{\footnotesize
			\begin{equation} \label{equation, Aviles exact sequence}
				0 \rightarrow    \Ker({\epsilon}_{K/F}) \rightarrow  H^1(G,U_K)   \xrightarrow{\lambda}  \bigoplus_{v \in M_F^0} H^1(G_w,U_{K_{w}}) \rightarrow \frac{\Cl(K)^G_{\trans}}{\epsilon_{K/F}(\Cl(F))} \rightarrow  0.
		\end{equation}}
\end{proposition}

\begin{remark} \label{remark, transgression map in Aviles paper}
	The composite  $\Cl(K)^G \rightarrow H^1(G,P(K)) \hookrightarrow H^2(G,U_K)$ is called the  \textit{transgression map} which we denote it by $\trans_{K/F}$. Also
	 $\Cl(K)^G_{\trans}$, the kernel of $\trans_{K/F}$, is called \textit{the group of transgressive ambiguous classes}, see \cite[$\S$2]{Aviles}. It must be noted that the transgression map is defined in a more general case as $\Trans_{K/F}:H^1(G_K,U_{\overline{K}})^G \rightarrow H^2(G,U_K)$, see \cite[$\S$3]{Serre} or \cite[$\S$4]{Yu} for the definition (We call both as transgression map and use different notations to avoid confusion).
\end{remark}

Using Proposition \ref{proposition, Aviles results}, we find the following description of relative P\'olya group and Ostrowski quotient.

\begin{corollary} \label{corollary, Po(L/K) coincides with Cl(L)_trans}
	For $K/F$ a finite Galois extension of number fields with Galois group $G$, we have
	\begin{equation*}
		\Po(K/F)=\Cl(K)^G_{\trans} \quad \text{and} \quad \Ost(K/F) \simeq \Coker(\lambda),
	\end{equation*}
	where 
	\begin{equation*}
		\lambda: H^1(G,U_K)   \rightarrow  \bigoplus_{\substack{v \in M_F^0}} H^1(G_w,U_{K_{w}})
	\end{equation*}
	denotes the localization map as in the exact sequence \eqref{equation, Aviles exact sequence}.
\end{corollary}

\begin{proof}
	Since  $I(K)^G$ is generated by all the relative Ostrowski ideals $\Pi_{\mathfrak{p}^f}(K/F)$, see \cite[proof of Proposition 2.2]{Brumer-Rosen}, $I(K)^G/P(K)^G$ coincides with the relative P\'olya group $\Po(K/F)$. Using \eqref{equation, BRZ exact sequence} and Proposition \ref{proposition, Aviles results}, we conclude that $\Ost(K/F)$ is nothing else than the cokernel of the localization map $\lambda$.
\end{proof}

	\section{On Tate-Shafarevich groups} \label{section, Tate-Shafarevich group}
	
	Keep the same notations introduced in the previous section.
	Let $E/F$ be an elliptic curve defined over $F$. By the \textit{Mordell-Weil Theorem} the group of $F$-rational points of $E$, denoted by $E(F)$, is finitely generated. An important tool to study $E(F)$, is the \textit{Tate-Shafarevich group} $\Sha(E/F)$, which measures the failure of the \textit{Hasse local-global principle} for 
	curves that are \textit{principal homogeneous spaces} for $E/F$. More formally, 
	$\Sha(E/F)$ is defined to be the group of everywhere locally trivial elements of $H^1(G_F,E(\overline{F}))$ \cite[Chapter X]{Silverman}:
	\begin{equation} \label{equation, Tate-Shafarevich group of elliptic curve}
		\Sha(E/F):=\Ker\left(H^1(G_F,E(\overline{F})) \rightarrow \bigoplus_{v \in M_F} H^1(G_{F_v},E(\overline{F_v}) \right).
	\end{equation}
	
	The famous conjecture of Tate and Shafarevich states that $\Sha(E/K)$ is finite \cite[Conjecture 4.13]{Silverman}. Though it has been shown that if the analytic rank of $E$ is smaller that $2$, then $\Sha(E/K)$ is finite \cite[Theorem 5.12]{Rubin}, but the general case has not yet been proved.
Using the cohomological analogy between the Mordell-Weil group $E(F)$ and the unit group $U_F$,  Tate-Shafarevich group of $F$ is defined similarly. 
	
	

	\begin{definition} \cite[$\S$1]{Schoof} \label{definition, Tate-Shafarevich group of number fields}
		The ``Tate-Shafarevich group $\Sha(F)$'' of $F$ is defined as 
		\begin{equation} \label{equation, Tate-Shafarevich group of K}
			\Sha(F):= \Ker\left(H^1(G_F,U_{\overline{F}}) \rightarrow \bigoplus_{v \in M_F} H^1(G_{\overline{F}_v},U_{\overline{F}_v}) \right),
		\end{equation}
		where for non-archimedean $v$ (resp. archimedean $v$), $U_{\overline{F}_v}$ denotes the valuation ring in $\overline{F}_v$ (resp. $U_{\overline{F}_v}=\overline{F}_v^{\times}$).
		Note that by Hilbert's Theorem 90,  $\Sha(F)$ can be defined just in term of the non-archimedean places of $F$.
	\end{definition}

	Despite Tate-Shafarevich conjecture for elliptic curves is still open, the situation for number fields is more understandable.

	\begin{proposition} \cite[Proposition 1]{Schoof} \label{proposition, the Tate-Shafarevich group of K is isomorphic to the ideal class group of K}
		$\Sha(F)$ is canonically isomorphic to the ideal class group of $F$. In particualr, $\Sha(F)$ is finite.
	\end{proposition}

 Replacing $\overline{F}$ with a finite Galois extension of $F$, one obtains the following notion.

	\begin{definition} \cite[$\S$1]{Schoof}
		Let $K/F$ be a finite Galois extension of number fields with Galois group $G$. The \textit{locally trivial cohomology group} $H^1_{lt}(G,U_K)$ is defined as 
		\begin{equation} \label{equation, definition of H^1_lt}
			H^1_{lt}(G,U_K):=\Ker \left(H^1(G,U_K) \rightarrow \bigoplus_{v \in M_F} H^1(G_{w},U_{K_{w}})\right),
		\end{equation}
	where for each $v \in M_F$, $w$ is a fixed place of $K$ lying above $v$.
	\end{definition}

	\begin{proposition} \cite[Corollary of Proposition 1]{Schoof} \label{proposition, Schoof-Washington result that Ker(epsilon) is isomorphic to H^1_lt}
		Let $K/F$ be a finite Galois extension of number fields with Galois group $G$. Then
		\begin{equation} \label{equation, isomorphism between kernel of epsilon and H^1_lt}
			\Ker(\epsilon_{K/F}) \simeq  H^1_{lt}(G,U_K),
		\end{equation}
	where $\epsilon_{K/F}:\Cl(F) \rightarrow \Cl(K)^G$ denotes the capitulation map.
	\end{proposition}

	\begin{remark}
		Note that isomorphism \eqref{equation, isomorphism between kernel of epsilon and H^1_lt} is also obtained from Gonz\'alez-Avil\'es exact sequence \eqref{equation, Aviles exact sequence}.
	\end{remark}

	\section{Relative Polya group as image of the restriction map} \label{section, Yu commutative diagram for number fields}
	As before, we use the notations introduced in Section \ref{section, Localization and transgression maps}.
	So let $K/F$ be a finite Galois extension of number fields with Galois group $G$. Replacing unit groups with rational points of elliptic curves, considered in Yu's paper \cite{Yu}, one obtains the following commutative diagram: 
	\begin{equation} \label{equation,  commutative diagram for unit groups}
		{\scriptsize
			\begin{tikzcd}
				0 \arrow[r] &	H^1(G,U_K) \arrow[r, "\Inf_{K/F}"] \arrow[d,"\overline{\mathcal{F}}"]
				& H^1(G_F,U_{\overline{F}}) \arrow[d, "\overline{\mathcal{G}}"] \arrow[r, "\Res_{K/F}"] & \Image(\Res_{K/F}) \arrow[d, "\overline{\mathcal{H}}"] \arrow[r]  & 0\\
				0 \arrow[r] &	\bigoplus_{v \in M_F} H^1(G_w,U_{K_w}) \arrow[r]
				& \bigoplus_{v \in M_F} H^1(G_{F_v},U_{\overline{F}_v}) \arrow[r] & \bigoplus_{v \in M_F} H^1(G_{K_w},U_{\overline{K}_w}),
			\end{tikzcd}
		} 
	\end{equation}
	where $\Inf_{K/F}$ denotes the inflation map and $\Res_{K/F}:H^1(G_F,U_{\overline{F}}) \rightarrow H^1(G_K,U_{\overline{K}})^G$ denotes the restriction map. 
	Using the snake lemma, we find the following exact sequence:
		\begin{equation} \label{equation, exact sequence obtained from commutative diagram for unit groups}
			0 \rightarrow    \Ker(\overline{\mathcal{F}}) \rightarrow \Ker(\overline{\mathcal{G}}) \rightarrow \Ker(\overline{\mathcal{H}}) \rightarrow \Coker(\overline{\mathcal{F}}) \rightarrow \overline{\mathcal{I}}\left(\Coker(\overline{\mathcal{F}})\right) \rightarrow 0,
	\end{equation}
	where 
	\begin{equation} \label{equation, map I-bar in commutative diagram for unit groups}
		\overline{\mathcal{I}}:\Coker(\overline{\mathcal{F}}) \rightarrow \Coker(\overline{\mathcal{G}})
	\end{equation}
	is the map obtained from diagram \eqref{equation,  commutative diagram for unit groups}. By Propositions \ref{proposition, Schoof-Washington result that Ker(epsilon) is isomorphic to H^1_lt} and \ref{proposition, the Tate-Shafarevich group of K is isomorphic to the ideal class group of K} we have $ \Ker(\overline{\mathcal{F}}) \simeq \Ker(\epsilon_{K/F})$ and $\Ker(\overline{\mathcal{G}}) \simeq \Cl(F)$, respectively. Hence \eqref{equation, exact sequence obtained from commutative diagram for unit groups} is equivalent to the following exact sequence:
		\begin{equation} \label{equation, exact sequence obtained from commutative diagram for unit groups in term of class group}
			0 \rightarrow  \Ker(\epsilon_{K/F})   \rightarrow \Cl(F) \rightarrow \Ker(\overline{\mathcal{H}}) \rightarrow \Coker(\overline{\mathcal{F}}) \rightarrow \overline{\mathcal{I}}\left(\Coker(\overline{\mathcal{F}})\right) \rightarrow 0.
	\end{equation}

	\begin{lemma} \label{lemma, Ker(g bar)=Cl(F), Ker(F bar)=Ker(epsilon) and Ker(H bar)=Po(K/F)}
		With the notations of this section, one has
			$\Ker(\overline{\mathcal{H}}) \simeq \Po(K/F)$.
	\end{lemma}
	
	\begin{proof}
		In \cite[$\S$4, Theorem 2]{Serre} it is proved that $\Image(\Res_{K/F})=\Ker(\Trans_{K/F})$, where 
		$\Trans_{K/F}:H^1(G_K,U_{\overline{K}})^G \rightarrow H^2(G,U_K)$ denotes the transgression map. Hence by Proposition \ref{proposition, the Tate-Shafarevich group of K is isomorphic to the ideal class group of K}, we have
		\begin{equation*}
		\Ker(\overline{\mathcal{H}})=\Ker(\Trans_{K/F}) \cap \Cl(K)^G=\Ker(\trans_{K/F})=\Cl(K)_{\trans}^G,
		\end{equation*}
	where $\Cl(K)_{\trans}^G$ denotes the group of transgressive ambiguous classes, see Remark \ref{remark, transgression map in Aviles paper}. Corollary \ref{corollary, Po(L/K) coincides with Cl(L)_trans} completes the proof.
	\end{proof}

	\begin{theorem} \label{theorem, the map I is zero for number field case}
		Let $K/F$ a finite Galois extension of number fields with Galois group $G$. For the localization maps
		\begin{equation*}
			\overline{\mathcal{F}} :H^1(G,U_K) \rightarrow \bigoplus_{v \in M_F} H^1(G_w,U_{K_w})
		\end{equation*}
		and
		\begin{equation*}
			\overline{\mathcal{G}}: H^1(G_F,U_{\overline{F}}) \rightarrow  \bigoplus_{v \in M_F} H^1(G_{F_v},U_{\overline{F}_v}),
		\end{equation*}
		the map
		\begin{equation*}
			\overline{\mathcal{I}}:\Coker(\overline{\mathcal{F}}) \rightarrow \Coker(\overline{\mathcal{G}})
		\end{equation*} 
		obtaining from diagram \eqref{equation,  commutative diagram for unit groups},
		is the zero map.
	\end{theorem}
	
	\begin{proof}
		By Lemma \ref{lemma, Ker(g bar)=Cl(F), Ker(F bar)=Ker(epsilon) and Ker(H bar)=Po(K/F)}, one can rewrite exact sequence \eqref{equation, exact sequence obtained from commutative diagram for unit groups in term of class group} as 
			\begin{equation*} \label{equation, exact sequence equivalent to the one obtained from diagram for unit groups}
				0 \rightarrow    \Ker(\epsilon_{K/F}) \rightarrow \Cl(F) \rightarrow  \Po(K/F) \rightarrow \Coker(\overline{\mathcal{F}}) \rightarrow \overline{\mathcal{I}}\left(\Coker(\overline{\mathcal{F}})\right) \rightarrow 0.
		\end{equation*}
	Equivalently, the sequence
		\begin{equation*} 
		0 \rightarrow    \epsilon_{K/F}\left(\Cl(F)\right) \rightarrow  \Po(K/F) \rightarrow \Coker(\overline{\mathcal{F}}) \rightarrow \overline{\mathcal{I}}\left(\Coker(\overline{\mathcal{F}})\right) \rightarrow 0
	\end{equation*}
is exact. By Corollary \ref{corollary, Po(L/K) coincides with Cl(L)_trans} we have $\Ost(K/F) \simeq \Coker(\overline{\mathcal{F}})$, hence $\overline{\mathcal{I}}$ must be the zero map.
	\end{proof}

\begin{remark}
Note that since $\Coker(\overline{\mathcal{F}}) \simeq \frac{\Cl(K)^G_{\trans}}{\epsilon_{K/F}(\Cl(F))}$, Theorem \ref{theorem, the map I is zero for number field case} shows that every class in $\Cl(K)^G_{\trans}$, modulo $\Image(\epsilon_{K/F})$, maps into the image of the localization map $	\overline{\mathcal{G}}$. So one may investigate to find a ``direct proof'' of Theorem \ref{theorem, the map I is zero for number field case}, which seems an interesting problem.
\end{remark}


\section{Main Results}	 \label{section, Main Results}
	
\textbf{Convention.} We keep the same notations introduced in Section \ref{section, Localization and transgression maps}. Throughout this section, $E/F$ denotes an elliptic curve defined over $F$ for which the Tate-Shafarevich groups $\Sha(E/F)$ and $\Sha(E/K)$ are finite.

	We recall that the main goal of this paper is to find elliptic curve analogue of Ostrowski quotient. To achieve this, we use two analogies between number fields and elliptic curves. The first one is the cohomological analogy between  unit groups of number fields and Mordell-Weil groups of elliptic curves. The other is the fact that the multiplicative group $\mathbb{G}_{m,F}$, i.e., the trivial one-dimensional algebraic torus over $F$, and elliptic curves (two very different \textit{algebraic groups} over $F$) do have some things in common. In particular, both of them admit \textit{N\'eron models} over $\mathcal{O}_F$ \footnote{The N\'eron model of $\mathbb{G}_{m,F}$ over $\mathcal{O}_F$ is described in \cite[Example 5, p. 291]{BLR}; also good, i.e., not too technical, references for N\'eron models of elliptic curves are \cite[VII, $\S \S$1,2,5]{Silverman} and \cite[IV, $\S$9]{Silverman2}.} (However, these models are very different from each other). 
	\begin{remark}
		Note that $\mathbb{G}_{m,F}$ is not \textit{unique} one-dimensional algebraic torus over $F$. There is another one, namely the norm one $F$-torus associated with a quadratic Galois extension $K/F$, i.e., the kernel $R_{K/F}^{(1)}\mathbb{G}_{m,K}$ of the norm map (for $F$-tori) $R_{K/F}\mathbb{G}_{m,K}\rightarrow \mathbb{G}_{m,F}$, where $R_{K,F}$ denotes Weil restriction. This $F$-torus $R_{K/F}^{(1)}\mathbb{G}_{m,K}$ has dimension $[K:F]-1=1$, but is not equal to $\mathbb{G}_{m,F}$.
	\end{remark}
	In the following, we use  Corollary \ref{corollary, Po(L/K) coincides with Cl(L)_trans} to define Ostrowski quotient of $\mathbb{G}_{m,F}$. Then using the same approach, we define \textit{Ostrowski quotient} $\Ost(E,K/F)$ for $E/F$ an elliptic curve defined over $F$. Due to some fundamental complications in this new setting, see Remark \ref{remark, complication for Ost(E,K/F)} below, we also define the notion of ``\textit{coarse Ostrowski quotient} $\Ost_c(E,K/F)$'' for which we do not take into account primes of bad reduction. Our main result, namely Theorem \ref{Main Theorem}, reveals some information about the structure of $\Ost_c(E,K/F)$. Finally, we investigate $\Ost_c(E,K/F)$ for some specific classes of elliptic curves $E$ and some extensions $K/F$.

\subsection{On elliptic curve analogue of the map $\overline{\mathcal{I}}$} As the first step, we need to investigate elliptic analog of the map $\overline{\mathcal{I}}$ as in \eqref{equation, map I-bar in commutative diagram for unit groups}. We begin with a special version of the well-known \textit{Global Duality Theorem}.

\begin{theorem}  \cite[Theorem 1]{Yu} \label{theorem, global duality theorem restated by Yu}
For $E/F$ an elliptic curve defined over $F$,  the following sequence is exact:
	\begin{equation} \label{equation, Yu exact sequence}
		0 \rightarrow \Sha(E/F) \rightarrow H^1(G_F,E) \xrightarrow{\mathcal{G}} \bigoplus_{v \in M_F} H^1(G_{F_v},E) \rightarrow \widehat{E(F)}^* \rightarrow 0,
	\end{equation}
	where $\widehat{E(F)}$ denotes the completion of $E(F)$ with respect to the topology defined by the subgroups of finite index, and $\widehat{E(F)}^*$ denotes the group of continuous characters of finite order of $\widehat{E(F)}$, i.e. $\widehat{E(F)}^*=\Hom_{\cts}(\widehat{E(F)},\mathbb{Q}/\mathbb{Z})$. 
\end{theorem}

\begin{remark}
	The last term in the original form of the exact sequence \eqref{equation, Yu exact sequence} appears based on the \textit{dual} of abelian varieties \cite[Theorem 1]{Yu}, whereas in the exceptional case the dual of an elliptic curve E will be isomorphic to $E$ itself. More percisely, in the characteristic zero, the dual of an abelian variety $A$ is isomorphic to the ``degree zero divisor group $\Pic^0(A)$'' of $A$ \cite[$\S$8]{Mumford}, and for an elliptic curves $E$ one has $\Pic^0(E) \simeq E$
	\cite[Chapter III, Proposition 3.4]{Silverman}.
\end{remark}

Similar to number field case, as discussed in Section \ref{section, Yu commutative diagram for number fields}, and following Yu \cite[$\S$2]{Yu}, consider the commutative diagram
\begin{equation} \label{equation, Yu commutative diagram}
	{\scriptsize
		\begin{tikzcd}
			0 \arrow[r] &	H^1(G,E(K)) \arrow[r, "\Inf_{K/F}^E"] \arrow[d,"\mathcal{F}"]
			& H^1(G_F,E) \arrow[d, "\mathcal{G}"] \arrow[r, "\Res^E_{K/F}"] & \Image(\Res^E_{K/F}) \arrow[d, "\mathcal{H}"] \arrow[r]  & 0\\
			0 \arrow[r] &	\bigoplus_{v \in M_F} H^1(G_{w},E(K_w) \arrow[r, "\eta"]
			& \bigoplus_{v \in M_F} H^1(G_{F_v},E) \arrow[r, "\theta"] & \bigoplus_{v \in M_F} H^1(G_{K_{w}},E),
		\end{tikzcd}
	}
\end{equation}
where $\Inf_{K/F}^E$ denotes the inflation map and 
\begin{equation} \label{equation, restriction map for E/K and E/L}
	\Res_{K/F}^E:H^1(G_F,E) \rightarrow H^1(G_K,E)^G
\end{equation}
denotes the restriction map. Using the snake lemma, we find the sequence
\begin{equation} \label{equation, exact sequence obtained from Yu commutative diagram}
	0 \rightarrow    \Ker(\mathcal{F}) \rightarrow \Sha(E/F) \xrightarrow{\widetilde{\Res}_{K/F}^E}  \Ker(\mathcal{H}) \rightarrow \Coker(\mathcal{F}) \rightarrow \mathcal{I}\left(\Coker(\mathcal{F})\right) \rightarrow 0,
\end{equation}
is exact, where 
\begin{equation*} 
	\widetilde{\Res}_{K/F}^E:=\Res_{K/F}^E\mid_{\Sha(E/F)}: \Sha(E/F) \rightarrow H^1(G_K,E)^G
\end{equation*}
and
\begin{equation} \label{equation, map I in Yu paper}
	\mathcal{I}:\Coker(\mathcal{F}) \rightarrow \Coker(\mathcal{G})
\end{equation}
is the map obtained from the above diagram. Let
\begin{equation} \label{equation, transgression map in Yu papaer}
	\Trans^E_{K/F}: H^1(G_K,E)^G \rightarrow H^2(G,E(K))
\end{equation}
be the transgression map, see Remark \ref{remark, transgression map in Aviles paper}. From \cite[$\S$4, Theorem 2]{Serre} we have
	\begin{equation} \label{equation, Image(ResK/FE)=Ker(TransK/FE)}
		\Image\left(\Res_{K/F}^E\right)=\Ker\left(\Trans^E_{K/F}\right).
	\end{equation}


\begin{lemma} \label{lemma, isomorphism between Ker(epsilon-tilda) and Ker(F)}
	With the notations of this section, 
	\begin{equation} \label{equation, image of res-tilda is contained in Sha(E/L)^G}
		\Image(\widetilde{\Res}_{K/F}^E) \subseteq \Sha(E/K)^G,
	\end{equation}
	and
	\begin{equation} \label{equation, Ker(epsilon-tilda) and Ker(F) are isomorphic}
		\Ker(\widetilde{\Res}_{K/F}^E) \simeq \Ker(\mathcal{F}).
	\end{equation}
\end{lemma}

\begin{proof}
		 Using  exact sequence \eqref{equation, exact sequence obtained from Yu commutative diagram}, equation \eqref{equation, Image(ResK/FE)=Ker(TransK/FE)} and  equality
	\begin{equation} \label{equation, kernel of H}
		\Ker(\mathcal{H}) = \Sha(E/K)^G \cap \Ker(\Trans^E_{K/F}),
	\end{equation}
	see \cite[page 213]{Yu}, we obtain the containment \eqref{equation, image of res-tilda is contained in Sha(E/L)^G}. Eexact sequence  \eqref{equation, exact sequence obtained from Yu commutative diagram} also implies that
	\begin{equation*}
		\Ker(\mathcal{F}) \simeq \Ker(\widetilde{\Res}_{K/F}^E). 
	\end{equation*} 
\end{proof}

\begin{remark}
	By Lemma \eqref{lemma, isomorphism between Ker(epsilon-tilda) and Ker(F)}, the map $\widetilde{Res}_{K/F}^E$ can be written as
	\begin{equation} \label{equation, res-tilda}
		\widetilde{\Res}_{K/F}^E: \Sha(E/F) \rightarrow \Sha(E/K)^G,
	\end{equation}
	which is the analogue of the capitulation map $\epsilon_{K/F}:\Cl(F) \rightarrow \Cl(K)^G$ as in \cite[$\S$2]{Aviles}. Moreover, the isomorphism \eqref{equation, Ker(epsilon-tilda) and Ker(F) are isomorphic},
	may be thought as the analogue of Schoof-Washington result in Proposition \ref{proposition, Schoof-Washington result that Ker(epsilon) is isomorphic to H^1_lt}.
\end{remark}

Similar to  the transgression map $\trans_{K/F}$ mentioned in Remark \ref{remark, transgression map in Aviles paper}, let
\begin{equation} \label{equation, definition of trans-tilda for Tate-Shafarevich groups}
	\trans^E_{K/F}:=\Trans^E_{K/F}|_{\Sha(E/K)^G}:\Sha(E/K)^G \rightarrow H^2(G,E(K)).
\end{equation}
By equation \eqref{equation, kernel of H}, we have
\begin{equation} \label{equation, ker(H) is equal to ker(trans-tilda)}
	\Ker(\mathcal{H})=\Ker(\trans^E_{K/F}),
\end{equation}
and exact sequence \eqref{equation, exact sequence obtained from Yu commutative diagram} yields
\begin{equation} \label{equation, image of res-tilda is contained into ker(trans-tilda)}
	\Image(\widetilde{\Res}^E_{K/F}) \subseteq \Ker(\trans^E_{K/F}).
\end{equation}

\begin{theorem} \label{theorem, EC-BRZI}
Let $E/F$ be an elliptic curve defined over $F$ and $K/F$  a finite Galois extension of number fields with Galois group $G$. Then the following statements are equivalent.

\begin{itemize}
	\item[(i)] The sequence 
	{\small
		\begin{equation} \label{equation, EC-BRZI} 
			0 \rightarrow    \Ker\left(\widetilde{\Res}_{K/F}^E\right) \rightarrow  H^1\left(G,E(K)\right)   \xrightarrow{\mathcal{F}}  \bigoplus_{v \in M_F} H^1\left(G_w,E(K_w)\right) \rightarrow \frac{\Ker(\trans^E_{K/F})}{\Image\left(\widetilde{\Res}_{K/F}\right)} \rightarrow  0
	\end{equation}}
is exact.
	\item[(ii)] 
	The map $\mathcal{I}:\Coker(\mathcal{F}) \rightarrow \Coker(\mathcal{G})$	as in \eqref{equation, map I in Yu paper}, is the zero map.
	
	\item[(iii)] The localization map 
	\begin{equation} \label{equation, map F0prime for elliptic curves}
		\mathcal{F}_0:\widehat{H}^0(G,E(K)) \rightarrow \prod_{v \in M_F} \widehat{H}^0\left(G_w,E(K_w)\right)
	\end{equation}
	is the zero map. 
\end{itemize}
In particular, if $\widehat{E(F)}$, introduced in Theorem \ref{theorem, global duality theorem restated by Yu}, is trivial
then \eqref{equation, EC-BRZI} is exact.
\end{theorem}

\begin{proof}
By Lemma \eqref{lemma, isomorphism between Ker(epsilon-tilda) and Ker(F)}, we have 
	\begin{equation*}
		\Ker(\widetilde{\Res}_{K/F}^E) \simeq \Ker(\mathcal{F}).
	\end{equation*}
	Hence sequence \eqref{equation, EC-BRZI} is exact if and only if 
	\begin{equation*}
		\frac{\Ker(\trans^E_{K/F})}{\Image\left(\widetilde{\Res}_{K/F}\right)} \simeq \Coker(\mathcal{F}).
	\end{equation*}
	By exact sequence \eqref{equation, exact sequence obtained from Yu commutative diagram} and equality \eqref{equation, ker(H) is equal to ker(trans-tilda)}, the last isomorphism holds if and only if $ \mathcal{I}\left(\Coker(\mathcal{F})\right)=0$ (Thanks to Yu's computations \cite[$\S$2]{Yu}). Also in \cite[Lemma 5]{Yu} it is proved that $\# \Image(\mathcal{I})=\#\Image\left(\mathcal{F}_0\right)$ which implies that the last two assertions are equivalent.
	Finally by Theorem \eqref{theorem, global duality theorem restated by Yu}, one has
	\begin{equation*}
		\Image(\mathcal{I}) \subseteq \Coker(\mathcal{G}) \simeq \widehat{E(F)}^*.
	\end{equation*}
\end{proof}

\begin{remark} \label{remark, EC-BRZ stands for}
Assume that the equivalent assertions in Theorem \ref{theorem, EC-BRZI} hold. Then comparing \eqref{equation, EC-BRZI} with Gonz\'alez-Avil\'es exact sequence \eqref{equation, Aviles exact sequence}, which is equivalent to \eqref{equation, BRZ exact sequence}, one may consider  $\frac{\Ker(\trans^E_{K/F})}{\Image\left(\widetilde{\Res}_{K/F}\right)}$ as elliptic curve analogue of the Ostrowski quotient $\Ost(K/F)$. Although, the assertions (i)--(iii) of Theorem \ref{theorem, EC-BRZI} may not hold in general, see Example \eqref{example, EC-BRZ is not exact} below. The notion of Ostrowski quotient for $E$ relative to $K/F$, in general case, will be defined in Section \ref{section, Ostrowski quotient for elliptic curves}.
\end{remark}

	\begin{corollary} \label{corollary, Ker(I) is isomorphic to Ostrowski quotient for elliptic curves}
		With the assumptions and notations of Theorem \ref{theorem, EC-BRZI}, one has
	\begin{equation*}
		\Ker(\mathcal{I}) \simeq \frac{\Ker(\trans^E_{K/F})}{\Image\left(\widetilde{\Res}_{K/F}\right)}.
	\end{equation*}
\end{corollary}

\begin{proof}
Using exact sequence \eqref{equation, exact sequence obtained from Yu commutative diagram} and Lemma \ref{lemma, isomorphism between Ker(epsilon-tilda) and Ker(F)} we find the sequence
	{\small
		\begin{equation*}
			0 \rightarrow  \Ker\left(\widetilde{Res}_{K/F}^E\right) \rightarrow \Sha(E/F) \rightarrow \Ker(\trans^E_{K/F}) \rightarrow \Coker(\mathcal{F}) \xrightarrow{\mathcal{I}} \Image(\mathcal{I}) \rightarrow 0,
	\end{equation*}	}
	is exact. Equivalently, we have the following exact sequence
	{\small
		\begin{equation*}
			0 \rightarrow  \Image\left(\widetilde{Res}_{K/F}^E\right)  \rightarrow  \Ker(\trans^E_{K/F}) \rightarrow \Coker(\mathcal{F}) \xrightarrow{\mathcal{I}} \Image(\mathcal{I}) \rightarrow 0,
	\end{equation*}	}
		and the assertion is proved. 

\end{proof}

\begin{corollary} \label{corollary, if the norm map is surjective then the analogue of  BRZ is exact}
	With the assumptions and notations of Theorem \ref{theorem, EC-BRZI}, if the norm map
	\begin{align*}
		N_{K/F}^E: E(K) & \rightarrow E(F) \\
		P & \mapsto  \sum_{\sigma \in G} \sigma(P)
	\end{align*}
	is surjective, e.g. if $E(F)=0$, then 
	\begin{itemize}
		\item[(1)] 
		sequence \eqref{equation, EC-BRZI} is exact;
		\item[(2)] for $G$ cyclic,  $\Ker(\trans^E_{K/F})=\Sha(E/K)^G$.
	\end{itemize}

\end{corollary}

\begin{proof}
	(1)	Since the norm map $N_{K/F}^E$ is surjective, we have
	\begin{equation*}
		\widehat{H}^0(G,E(K))=\frac{E(F)}{N_{K/F}^E(E(K))}=0.
	\end{equation*}
	Hence the localization map $\mathcal{F}_0$  \eqref{equation, map F0prime for elliptic curves}
	is zero and exactness of \eqref{equation, EC-BRZI} follows from Theorem \ref{theorem, EC-BRZI}.
	
	(2) Since $K/F$ is cyclic, we have
	\begin{equation*}
		H^2(G,E(K))=\widehat{H}^0(G,E(K))=0.
	\end{equation*}
	Therefore the transgression map $\trans^E_{K/F}$ \eqref{equation, definition of trans-tilda for Tate-Shafarevich groups} is the zero map.
\end{proof}

\begin{remark} \label{Remark, if the norm map on units is surjective}
Applying the same above reasoning for the transgression map $\trans_{K/F}$, appeared in Remark \ref{remark, transgression map in Aviles paper},
 one can obtain the number field analogue of part (2) of Corollary \ref{corollary, if the norm map is surjective then the analogue of  BRZ is exact}, i.e., if $K/F$  is cyclic and  the norm map $N_{K/F}:U_K \rightarrow U_F$ is surjective, then $\Po(K/F)=\Cl(K)^G$. Note that by Corollary \ref{corollary, Po(L/K) coincides with Cl(L)_trans}, $\Po(K/F)=\Ker\left(\trans_{K/F}\right)$.
\end{remark}

\begin{corollary} \label{corollary, order of the quotient of ker(trans-tilda) over Sha(E/K)}
	With the assumptions and notations of Theorem \ref{theorem, EC-BRZI}, if sequence \eqref{equation, EC-BRZI} is exact, then
	\begin{equation} \label{equation, order of ker(trans-tilda)}
	\frac{\# \Ker(\trans^E_{K/F})}{\# \Sha(E/F)}=\frac{\prod_{v \in M_F} \# H^1\left(G_{w},E(K_{w})\right)}{\# H^1\left(G,E(K)\right)}.
\end{equation}
Furthermore, $\trans^E_{K/F}$ is the zero map if and only if
\begin{equation*}
\frac{\# \Sha(E/K)^G}{\# \Sha(E/F)}=\frac{\prod_{v \in M_F} \# H^1\left(G_{w},E(K_{w})\right)}{\# H^1\left(G,E(K)\right)}.
\end{equation*}
\end{corollary}

\begin{proof}
	Immediately follows from Theorem \ref{theorem, EC-BRZI}.
\end{proof}


As a generalization of the main theorem in \cite{Aviles2000}, Yu obtained the following result.

\begin{proposition} \cite[Corollary 7]{Yu} \label{proposition, Yu result for generalization of Aviles result}
	Let $E/F$ be an elliptic curve defined over $F$ and $K/F$  a finite Galois extension of number fields with Galois group $G$. If
	\begin{equation*}
		\widehat{H}^0(G,E(K))=H^2\left(G,E(K)\right)=0,
	\end{equation*}
	then
	\begin{equation*}
			\frac{\# \Sha(E/K)^G}{\# \Sha(E/F)}=\frac{\prod_{v \in M_F} \# H^1\left(G_{w},E(K_{w})\right)}{\# H^1\left(G,E(K)\right)}.
	\end{equation*}
\end{proposition}

\begin{remark} \label{remark, our corollary is a generalization ofresults of Yu and Aviles}
	If $\widehat{H}^0(G,E(K))=0$, then the localization map $\mathcal{F}_0$ \eqref{equation, map F0prime for elliptic curves} is zero, so \eqref{equation, EC-BRZI} is exact by Theorem \ref{theorem, EC-BRZI}. Also 
	triviality of $H^2\left(G,E(K)\right)$ implies that the map  $\trans^E_{K/F}$ \eqref{equation, definition of trans-tilda for Tate-Shafarevich groups} is the zero map. Whereas, the converse of these statements may not hold in general. Therefore,  Corollary \eqref{corollary, order of the quotient of ker(trans-tilda) over Sha(E/K)} can be thought  as a generalization of Proposition \eqref{proposition, Yu result for generalization of Aviles result} and \cite[Main Theorem]{Aviles2000} (It would be a challenging task to find some examples to show that this is a \textit{strict} generalization). Consequently, the family of elliptic curves suggested by Gonz\'alez-Avil\'es \cite[$\S$1]{Aviles2000} satisfy in equivalent assertions of Theorem \ref{theorem, EC-BRZI}. In other words, for $E/\mathbb{Q}$ an elliptic curve defined over $\mathbb{Q}$ given by a Weierstrass equation of negative discriminant and $K/\mathbb{Q}$  a finite Galois extension for which $E(K)$ is finite of order prime to $[K:\mathbb{Q}]$, \eqref{equation, EC-BRZI} is exact. In contrast, one can find examples of elliptic curves over finite Galois extensions of number fields, for which  \eqref{equation, EC-BRZI} is not exact, see Example \ref{example, EC-BRZ is not exact} below. 
\end{remark}

\subsection{Ostrowski quotient of $\mathbb{G}_{m,F}$} \label{subsection, Ostrowski quotient of Gm,F}
As mentioned before, the multiplicative group $\mathbb{G}_{m,F}$ and the elliptic curve $E$ have N\'eron models over $\mathcal{O}_F$. Although these models are different, yet this is a key point to define the Ostrowski quotient for these two algebraic groups over $F$. First consider $\mathbb{G}_{m,F}$ and
denote its N\'eron model over $\mathcal{O}_F$ by $\mathcal{G}_m$. In \cite[Example 1, p.291]{BLR} it is proved that the identity component of $\mathcal{G}_m$ is $\mathcal{G}_m^0=\mathbb{G}_{m,\mathcal{O}_F}$. Then the canonical sequence of $\mathcal{O}_F$-group schemes
\begin{equation*}
	0 \rightarrow  \mathcal{G}_m^0 \rightarrow \mathcal{G}_m \rightarrow \mathcal{G}_m/\mathcal{G}_m^0 \rightarrow 0
\end{equation*}
is
\begin{equation} \label{equation, exact sequnece from gm to sum ipstar}
	0 \rightarrow \mathbb{G}_{m,\mathcal{O}_F} \rightarrow \mathcal{G}_m \rightarrow \bigoplus_{\mathfrak{p} \in \mathbb{P}_F} (i_{\mathfrak{p}})_*\mathbb{Z} \rightarrow 0,
\end{equation}

where, for each $\mathfrak{p} \in \mathbb{P}_F$, the morphism $i_{\mathfrak{p}}:\text{Spec}\, k(\mathfrak{p}) \rightarrow \text{Spec}\, \mathcal{O}_F$ is induced by the canonical projection $\mathcal{O}_F \rightarrow k(\mathfrak{p})=\mathcal{O}_F/\mathfrak{p}$. Since $\mathcal{G}_m^0(\mathcal{O}_F)=\mathbb{G}_{m,\mathcal{O}_F}(\mathcal{O}_F)=U_F$ and $\mathcal{G}_m(\mathcal{O}_F)=F^{\times}$, the \'etale $\mathcal{O}_F$-cohomology sequence induced by \eqref{equation, exact sequnece from gm to sum ipstar} yields the exact sequence of abelian groups 
\begin{equation} \label{equation, exact sequence sumZ to etale H1}
	0 \rightarrow U_F \rightarrow F^{\times} \rightarrow \bigoplus_{\mathfrak{p} \in \mathbb{P}_F} \mathbb{Z}\rightarrow H_{\text{\'et}}^1(\mathcal{O}_F,\mathbb{G}_m) \rightarrow 0.
\end{equation}

The $0$ at the right comes from the fact that $H_{\text{\'et}}^1(\mathcal{O}_F,\mathcal{G}_m)$ injects into $H^1(F,\mathbb{G}_m)$, which is $0$ by Hilbert's Theorem 90. Since 
\begin{equation*}
	\bigoplus_{\mathfrak{p} \in \mathbb{P}_F} \mathbb{Z} \simeq I(F) \quad \text{and} \quad H_{\text{\'et}}^1(\mathcal{O}_F,\mathbb{G}_m) \simeq \text{Pic}(\mathcal{O}_F) \simeq \Cl(F),
\end{equation*}
the last sequence is simply the well-known exact sequence 
\begin{equation*}
	0 \rightarrow U_F \rightarrow F^{\times} \rightarrow I(F) \rightarrow \Cl(F) \rightarrow 0.
\end{equation*}

Working over $K$ now, we similarly obtain an exact sequence of $G$-modules
\begin{equation} \label{equation, exact sequence UK to K* to IK to ClK}
	0 \rightarrow U_K \rightarrow K^{\times} \rightarrow I(K) \rightarrow \Cl(K) \rightarrow 0.
\end{equation}
The image of the middle map in \eqref{equation, exact sequence UK to K* to IK to ClK} is $P(K)$, so we obtain the short exact sequence of $G$-modules
\begin{equation} \label{equation, exact sequence UK to K* to PK}
	0 \rightarrow U_K \rightarrow K^{\times} \rightarrow P(K) \rightarrow 0.
\end{equation}
Taking $G$-cohomology of \eqref{equation, exact sequence UK to K* to PK}, we obtain an exact sequence of abelian groups
\begin{equation} \label{equation, exact sequence F* to PKG to H1}
	F^{\times} \rightarrow P(K)^G \rightarrow H^1(G,U_K) \rightarrow 0.
\end{equation}
using Corollary \ref{corollary, Po(L/K) coincides with Cl(L)_trans}, one is then naturally led to study the cokernel of the map
\begin{equation} \label{equation, lambdaGm and K/F}
	\lambda_{\mathcal{G}_m^0,K/F}:H^1(G,U_K) \rightarrow \bigoplus_{v \in M_F} H^1(G_w,U_{K_w}),
\end{equation}
i.e., the \textit{Ostrowski quotient of $\mathbb{G}_{m,F}$ relative to $K/F$}:
\begin{equation} \label{equation, Ostrowski quotient of Gm,F}
	\Ost(\mathbb{G}_{m,F},K/F)=\Coker \lambda_{\mathcal{G}_m^0,K/F},
\end{equation}
which is, by exact sequence \eqref{equation, Aviles exact sequence}, equal to $\frac{\Cl(K)^G_{\trans}}{\epsilon_{K/F}(\Cl(F))}$.


\subsection{Ostrowski quotient for elliptic curves} \label{section, Ostrowski quotient for elliptic curves}
Using elliptic curve analogue of computations in Section \ref{subsection, Ostrowski quotient of Gm,F}, one can define the notion of Ostrowski quotient for $E$ relative to $K/F$. 
Consider an \textit{integral} equation for $E$, i.e., an equation for $E$ with coefficients in $\mathcal{O}_F$, and denote by $\mathcal{E}$ the N\'eron model of $E$ defined over $\mathcal{O}_F$\footnote{It is like a family of curves which includes, as one member of the family, the curve $E$ viewed over $F$ (i.e., $E$ with its original non-integral equation equation over $F$). All but finitely many of the members of this family are elliptic curves (defined over the residue fields $k(\mathfrak{p})$). The finitely many exceptions occur over the primes of bad reduction for $E$. Over every such bad prime $\mathfrak{p}$, when does not get an elliptic curve, but rather a finite union of connected linear groups of the form $T \times \mathbb{G}_{a,k(\mathfrak{p})}^n$, where $T$ is a $k(\mathfrak{p})$-torus, $\mathbb{G}_{a,k(\mathfrak{p})}^n$ is the additive group over $k(\mathfrak{p})$ and $n \geq 0$ is an integer.}. Then the analogue of sequence \eqref{equation, exact sequnece from gm to sum ipstar} is
\begin{equation} \label{equation, exact sequnece from mathcalE to sum ivstar}
	0 \rightarrow \mathcal{E}^0 \rightarrow \mathcal{E} \rightarrow \bigoplus_{v \in M_F^0} (i_v)_* \Phi_v(E) \rightarrow 0,
\end{equation}
where for each $v \in M_F^0$, $\Phi_v(E)$ is the (finite, algebraic) $k(v)$-group of components of $\mathcal{E}$ at $v$ (note that the above sum is finite since it is concentrated on the set of primes of bad reduction for $E$. This is contrast to the case of $\mathbb{G}_{m,F}$, since the sum is in \eqref{equation, exact sequnece from gm to sum ipstar} is infinite). Taking \'etale $\mathcal{O}_F$-cohomology of \eqref{equation, exact sequnece from mathcalE to sum ivstar}, we obtain the following analog of \eqref{equation, exact sequence sumZ to etale H1}
\begin{equation}
	0 \rightarrow E_0(F) \rightarrow E(F) \rightarrow \bigoplus_{v \in M_F^0} \Phi_v(E)(k(v)) \rightarrow H_{\text{\'et}}^1(\mathcal{O}_F,\mathcal{E}^0) \rightarrow \dots ,
\end{equation}
where $E_0(F):=\mathcal{E}^0(\mathcal{O}_F)$ is the elliptic analog of $U_F$. The next term in the above sequence, namely $H_{\text{\'et}}^1(\mathcal{O}_F,\mathcal{E})$, injects into $H^1(G_K,E)$, but the latter group is not $0$ anymore (no Hilbert 90 for an elliptic curve).

Working over $K$ now, let $\mathcal{E}^K$ denote the N\'eron model of the $K$-elliptic curve $E_K$. We similarly obtain an exact sequence of $G$-modules which is the elliptic analog of \eqref{equation, exact sequence UK to K* to IK to ClK}, namely
\begin{equation*}
	0 \rightarrow E_0(K) \rightarrow E(K) \rightarrow \bigoplus_{w \in M_K^0} \Phi_w(E_K)(k(w)) \rightarrow H_{\text{\'et}}^1(\mathcal{O}_K,(\mathcal{E}^K)^0) \rightarrow \dots ,
\end{equation*}
where $E_0(K):=(\mathcal{E}^K)^0(\mathcal{O}_K)$ plays the role of $U_K$. The analog of sequence \eqref{equation, exact sequence UK to K* to PK} is then the sequence
\begin{equation} \label{equation, exact sequence E0(K) to E(K) to its quotients}
	0 \rightarrow E_0(K) \rightarrow E(K) \rightarrow E(K)/E_0(K) \rightarrow 0.
\end{equation}

Thus, if we view $\oplus_{w \in M_K^0} \Phi_w(E_K)(k(w))$ as the analog of $I(K)$, then the image of $E(K)/E_0(K)$ in $\oplus_{w \in M_K^0} \Phi_w(E_K)(k(w))$ is the analog of $P(K)$ (but note that $E(K)/E_0(K)$ is a finite group). Taking $G$-cohomology of \eqref{equation, exact sequence E0(K) to E(K) to its quotients}, we obtain the elliptic analog of \eqref{equation, exact sequence F* to PKG to H1}, namely
\begin{equation*}
	E(F) \rightarrow (E(K)/E_0(K))^G \rightarrow H^1(G,E_0(K)) \rightarrow H^1(G,E(K)) \rightarrow \dots .
\end{equation*}
Now we see that the natural analog of the map \eqref{equation, lambdaGm and K/F} is the map
\begin{equation} \label{equation, lambda mathcalE0 and K/F}	
	\lambda_{\mathcal{E}^0,K/F}: H^1(G,E_0(K)) \rightarrow \bigoplus_{v \in M_F} H^1(G_w,E_0(K_w))
\end{equation}
where for $v$ archimedean, $E_0(K_w):=E(K_w)$ and $H^1(G,E_0(K)) \rightarrow H^1(G_w,E_0(K_w))$
is the composition
\begin{equation*}
	H^1(G,E_0(K)) \rightarrow H^1(G,E(K)) \rightarrow H^1(G_w,E(K_w)).
\end{equation*}
The map \eqref{equation, lambda mathcalE0 and K/F} fits into an exact and commutative diagram of abelian groups \footnote{The quotient groups $E(K_w)/E_0(K_w)$ are discussed in \cite[VII, $\S$6]{Silverman}.}
\begin{equation} \label{equation, commutative diagram for lambda  mathcalE0 and K/F}
	{\scriptsize
		\begin{tikzcd}
			\left(\frac{E(K)}{E_0(K)}\right)^G \arrow[r] \arrow[d] &	H^1(G,E_0(K)) \arrow[r] \arrow[d,"\lambda_{\mathcal{E}^0,K/F}"]
			& H^1(G,E(K)) \arrow[d, "\lambda_{\mathcal{E},K/F}"] \\
			\bigoplus_{v \in M_F} \left(\frac{E(K_w)}{E_0(K_w)}\right)^{G(w)} \arrow[r] &	\bigoplus_{v \in M_F} H^1(G_{w},E_0(K_w)) \arrow[r]
			& \bigoplus_{v \in M_F} H^1(G_{w},E(K_w)),
		\end{tikzcd}
	}
\end{equation}
where $G(w)=\text{Gal}(k(w)/k(v))$ if $v \in M_F^0$ and $G(w)=0$ otherwise. Thus the natural elliptic analog of the group $\Ost(\mathbb{G}_{m,F},K/F)$, introduced as in \eqref{equation, Ostrowski quotient of Gm,F}, is the cokernel of the middle vertical map in the above diagram.
\begin{definition} \label{definition, Ost(E,K/F)}
The group
\begin{equation} \label{equation, Ostrowski quotient of E,K/F}
	\Ost(E,K/F):=\Coker \, \lambda_{\mathcal{E}^0,K/F}
\end{equation}
is called the \textit{Ostrowski quotient of $E$ relative to $K/F$}.
\end{definition}
Now the canonical exact sequence 
{\small
	\begin{equation*} \label{equation, EC-BRZ exact sequence}
		\tag{EC-BRZ} 
		0 \rightarrow \Ker(\lambda_{\mathcal{E}^0,K/F}) \rightarrow H^1(G,E_0(K)) \rightarrow \bigoplus_{v \in M_F} H^1(G_w,E_0(K_w)) \rightarrow \Ost(E,K/F) \rightarrow 0
\end{equation*}}
can be thought as the elliptic analog of Gonz\'alez-Avil\'es exact sequence \eqref{equation, Aviles exact sequence} which is, by Corollary  \ref{corollary, Po(L/K) coincides with Cl(L)_trans}, equivalent to \eqref{equation, BRZ exact sequence}. 
Note that ``EC-BRZ'' stands for the \textit{Elliptic Curve analogue of} \eqref{equation, BRZ exact sequence}.
\begin{remark} \label{remark, complication for Ost(E,K/F)}
	Studying the structure of $\Ost(E,K/F)$ is (much) more difficult than studying that of $\Ost(\mathbb{G}_{m,F},K/F)$. In the case of $\Ost(E,K/F)$ one has to take into account the primes of bad reduction for $E$, whereas in the case of $\Ost(\mathbb{G}_{m,F},K/F)$ \textit{there are no primes of bad reduction}. Indeed, if $T$ is an algebraic torus over $F$ with N\'eron model $\mathcal{T}$ over $\mathcal{O}_F$, one says that $T$ has \textit{good reduction} at a prime $\mathfrak{p} \in \mathbb{P}_F$ if the connected algebraic $k(\mathfrak{p})$-group $\mathcal{T}_{k(\mathfrak{p})}^0$ is a $k(\mathfrak{p})$-torus. When $T=\mathbb{G}_{m,F}$, we have 
	\begin{equation*}
		\mathcal{T}_{k(\mathfrak{p})}^0=(\mathcal{G}_m^0)_{k(\mathfrak{p})}=(\mathbb{G}_{m,\mathcal{O}_F})_{k(\mathfrak{p})}=\mathbb{G}_{m,k(\mathfrak{p})},
	\end{equation*}
	which is certainly a $k(\mathfrak{p})$-torus for any $\mathfrak{p} \in \mathbb{P}_F$. So $\mathbb{G}_{m,F}$ has \textit{good reduction everywhere over $F$}. Consequently, the complications that arise in the study of $\Ost(\mathbb{G}_{m,F},K/F)$ come \textit{mainly from the primes of $F$ that ramify in $K$}. Further, $T=\mathbb{G}_{m,F}$ is just \textit{one} (very simple) $1$-dimensional algebraic group with very simple reductions everywhere, whereas the class of all elliptic curves over $F$ is an infinite family of $1$-dimensional algebraic groups whose members degenerate (i.e., have bad reduction) at some primes of $F$. The above comments clearly show that the gulf that separates the classical case $\Ost(\mathbb{G}_{m,F},K/F)$ from the elliptic case $\Ost(E,K/F)$ is \textit{truly vast}.
\end{remark}

Although $\Ost(E,K/F)$ is the ``correct'' analog of $\Ost(\mathbb{G}_{m,F},K/F)$, one should also investigate the cokernel of the right-hand vertical map in diagram \eqref{equation, commutative diagram for lambda  mathcalE0 and K/F}, i.e., the map $\lambda_{\mathcal{E},K/F}$ which is the same localization map $\mathcal{F}$ appeared in diagram \eqref{equation, Yu commutative diagram}.

\begin{definition} \label{definition, coarse Ostrowski quotient}
	The group
	\begin{equation} \label{equation, coarse Ostrowski quotient of E}
		\Ost_c(E,K/F)=\Coker \, \lambda_{\mathcal{E},K/F},
	\end{equation}
	is called the \textit{coarse Ostrowski quotient of $E$ relative to $K/F$} (``coarse'' because it does not take into account the primes of bad reduction for $E_K$, in contrast to the group $\Ost(E,K/F)$). 
\end{definition}
\begin{remark} \label{remark, Ostrowski and coarse Ostrowski may coincide}
Note that the analog of $\Ost_c(E,K/F)$ in the classical case is the \textit{trivial group} (because of Hilbert's Theorem 90). Also note that there exists instances where the groups  $\Ost(E,K/F)$ and  $\Ost_c(E,K/F)$ \textit{coincide}. For example, if $j(E) \in \mathbb{Z}$ (e.g., $E$ has complex multiplication \cite[VII, Exersice 7.10]{Silverman}), then there exists an extension $K/F$ as above such that $E/K$ has good reduction at all primes of $K$ \cite[VII, Proposition 5.5]{Silverman}, whence $E(K)=E_0(K)$ and (therefore)  $\Ost(E,K/F)=\Ost_c(E,K/F)$.  
\end{remark}

 The main result of this paper is the following non-trivial structure theorem for the group $\Ost_c(E,K/F)$.

\begin{theorem} \label{Main Theorem}
There exists a canonical exact sequence of abelian groups
	\begin{equation} \label{equation, exact sequence in Main Theorem}
		0 \rightarrow \frac{\Ker(\trans^E_{K/F})}{\Image\left(\widetilde{\Res}_{K/F}\right)} \rightarrow \Ost_c(E,K/F) \xrightarrow{\mathcal{I}} \widehat{E(F)}^*,
	\end{equation}
where $\widetilde{\Res}_{K/F}^E: \Sha(E/F) \rightarrow \Sha(E/K)^G$ and  $\trans^E_{K/F}:\Sha(E/K)^G \rightarrow H^2(G,E(K))$ denote the restriction map \eqref{equation, res-tilda} and
 the transgression map \eqref{equation, definition of trans-tilda for Tate-Shafarevich groups}, respectively. Also
\begin{equation*}
\mathcal{I}:\Ost_c(E,K/F) \rightarrow \Coker\left(H^1(G_F,E) \xrightarrow{\mathcal{G}} \bigoplus_{v \in M_F} H^1(G_{F_v},E)\right)
\end{equation*}
is the map \eqref{equation, map I in Yu paper}, and $\widehat{E(F)}^*=\Hom_{\cts}(\widehat{E(F)},\mathbb{Q}/\mathbb{Z})$ as introduced in Theorem \eqref{theorem, global duality theorem restated by Yu}.
\end{theorem}

\begin{proof}
Immediately follows from Corollary \ref{corollary, Ker(I) is isomorphic to Ostrowski quotient for elliptic curves}, since $\Image(\mathcal{I}) \subseteq \widehat{E(F)}^*$ by Theorem \ref{theorem, global duality theorem restated by Yu}.
\end{proof}

	\subsection{Investigating $\Ost_c(E,K/F)$ for quadratic extensions} \label{section, EC-BRZ for quadratic extensions}

In this section, for  the class of curves $E$ relative to quadratic extensions $K/F$, we use some results of Qiu \cite{Qiu} and Yu \cite{Yu}, to find some information about the left-hand group and the image of the map $\mathcal{I}$ in exact sequence \eqref{equation, exact sequence in Main Theorem}. Then as application of Theorem \ref{Main Theorem}, we give the structure of $\Ost_c(E,K/F)$ for some explicit examples of $(E,K/F)$ in this family.

\begin{theorem} \label{theorem, order of ker(trans-tilde)}
	For a number field $F$, let $E/F$ be an elliptic curve and $K=F(\sqrt{D})$ be a quadratic extension of $F$ with $Gal(K/F)=<\sigma>$. Denote by $E_D$ the quadratic $D$-twist of $E$. If the Tate-Shafarevich groups $\Sha(E/F)$, $\Sha(E_D/F)$ and $\Sha(E/K)$ are finite and  \eqref{equation, EC-BRZI} is exact, then
	\begin{equation} \label{equation, order of ker(transE,K/F)}
		\#  \Ker(\trans^E_{K/F})= \frac{\# \Sha(E/K). \left(E(F):N^E_{K/F}(E(K)) \right)}{\#\Sha(E_D/F)},
	\end{equation}
	where $N^E_{K/F}\left(E(K)\right) =\{P+ \sigma(P) : P \in E(K)\}$.	
\end{theorem}

\begin{proof}
	From  \eqref{equation, EC-BRZI} we have
	\begin{equation} \label{equation, order of trans-tilde from our exact sequence}
		\# \Ker(\trans^E_{K/F})=\frac{\# \Sha(E/F).\prod_{v \in M_F}  H^1\left(G_w,E(K_w)\right)}{\# H^1\left(G,E(K)\right)}.
	\end{equation}
	On the other hand, under finiteness assumption of Tate-Shafarevich groups, Yu \cite[Main Theorem]{Yu} proved
	\begin{equation} \label{equation, Yu's main theorem}
		\frac{\# \Sha(E/F). \# \Sha(E_D/F)}{\# \Sha(E/K)}=\frac{ \# \widehat{H}^0(G,E(K)). \# H^1(G,E(K))}{\prod_{v \in M_F} \# H^1\left(G_w,E(K_w)\right)}.
	\end{equation}
 Since $\#\widehat{H}^0(G,E(K))=\left(E(F):N^E_{K/F}(E(K))\right)$, using equations \eqref{equation, order of trans-tilde from our exact sequence} and \eqref{equation, Yu's main theorem} we obtain the desired equality.
\end{proof}

The example below shows that the exactness of  \eqref{equation, EC-BRZI}, equivalently $\Image(\mathcal{I})=0$ by Theorem \ref{theorem, EC-BRZI}, is a necessary condition in Theorem \ref{theorem, order of ker(trans-tilde)}.

\begin{example} \label{example, EC-BRZ is not exact}
	For the elliptic curve
	\begin{equation*}
		E/\mathbb{Q}: y^2=x^3-x+1/4
	\end{equation*}
	and the imaginary quadratic field $K=\mathbb{Q}(\sqrt{D})$ with
	\begin{align*}
		D \in \{&-7,-11,-47,-71,-83,-84,-127,-159,-164,-219,-231,\\
		&	-263,-271,-287,-292,-303,-308,-359,-371,-404,-443,-447,-471
		\},
	\end{align*}
	we have $\Sha(E/K)=\Sha(E_D/\mathbb{Q})=0$  and $\left(E(\mathbb{Q}):N^E_{K/\mathbb{Q}}(E(K))\right)=2$, see \cite[part (2) of Theorem 4.1]{Qiu}. In particular,  $ \Ker(\trans^E_{K/\mathbb{Q}})=0$ and  equality \eqref{equation, order of ker(transE,K/F)} doesn't hold. Hence by Theorem \ref{theorem, order of ker(trans-tilde)}, we conclude that for the given elliptic curve $E/\mathbb{Q}$ and the imaginary quadratic fields $K$, \eqref{equation, EC-BRZI} can not be exact. Theorem \ref{theorem, EC-BRZI} shows that the map
	\begin{equation*} 
		\mathcal{I}:\Ost_c(E,K/F) \rightarrow \Coker(\mathcal{G})
	\end{equation*}
is nonzero. Also since $ \Ker(\trans^E_{K/\mathbb{Q}})=0$, by Corollary \ref{corollary, Ker(I) is isomorphic to Ostrowski quotient for elliptic curves} the map $\mathcal{I}$ is injective. Now
	by \cite[Lemma 5]{Yu}, one has
	\begin{equation*}
		\# \Ost_c(E,K/F)=\# \Image(\mathcal{I}) \leq \# \widehat{H}^0(G,E(K))=  \left(E(\mathbb{Q}):N^E_{K/\mathbb{Q}}(E(K))\right)=2.
	\end{equation*}
Consequently,
	\begin{equation*}
	\Ost_c(E,K/F) \simeq \Image(\mathcal{I})  \simeq  \mathbb{Z}/2\mathbb{Z} \hookrightarrow \Coker(\mathcal{G}) \simeq \widehat{E(\mathbb{Q})}^*=\Hom_{\cts}(\widehat{E(\mathbb{Q})},\mathbb{Q}/\mathbb{Z}),
	\end{equation*}
where $\widehat{E(\mathbb{Q})}$ is defined as in Theorem \ref{theorem, global duality theorem restated by Yu}. Note that $E(\mathbb{Q}) \simeq \mathbb{Z}$, see \cite[proof of Theorem 4.1]{Qiu}. 
\end{example}

	\begin{proposition} \cite[$\S$1--2]{Qiu} \label{proposition, Qiu results}
Let $E/F$ be an elliptic curve defined over $F$, and $K=F(\sqrt{D})$ be a quadratic extension of $F$ with $G=Gal(K/F)=<\sigma>$. Denote by $E_D$ the quadratic $D$-twist of $E$. Also denote by $r_F$ and $r_{D,F}$ the $\mathbb{Z}$-rank of $E(F)$ and $E_D(F)$, respectively.
	The following assertions hold.
		\begin{itemize}
			
			
			\item[(i)] $\# H^1\left(G,E(K)\right)= 2^{r_{D,F}-r_F}.\left(E(F): N^E_{K/F}(E(K))\right)$.
			
			\vspace*{0.25cm}

			\item[(ii)] Assume that  $\Sha(E/F)$, $\Sha(E_D/F)$ and $\Sha(E/K)$ are finite. Also let
				\begin{equation*}
				S= \{v \in M_F^0 \,  : \, v \, \text{is ramified in} \, K \, \text{or} \, E \, \text{has bad reduction at} \, v \},
			\end{equation*}
			and
			\begin{equation*}
				S_0=\{v \in S : v \, \text{is ramified on inertial in} \, K \}.
			\end{equation*}
			Then
					\begin{itemize}
				\item[(ii-1)] 
				\begin{equation} \label{equation, Refined Yu's formula by Qiu}
					\frac{\# \Sha(E/F).  \# \Sha(E_D/F)}{\# \Sha(E/K)}=2^{r_{D,F}-r_F-\delta(E,F,K)}. \left(E(F):N(E(K))\right)^2;
				\end{equation}
				\item[(ii-2)]
				 		\begin{equation} \label{equation, order of local H0's}
				 		\prod_{v \in M_F} \#H^1(G_w,E(K_w))=\prod_{v \in S_0 \cup M_F^{\infty}} \#H^1(G_w,E(K_w))=2^{\delta(E,F,K)}.
				 	\end{equation} 
			\end{itemize}
			where $\delta(E,F,K)=\delta_{\infty} + \delta_f$ with
				\vspace*{0.15cm}
			\begin{itemize}
				\item 
				$\delta_{\infty}=\#\{v \in M_F^{\infty} : v \, \text{is ramified in} \, K \, \text{and} \, \sigma_{v}(\Delta(E)) >0\}$,
				
				\vspace*{0.15cm}
				
				\item 
				$\delta_f=\sum_{v \in S_0}\log_2\left(E(F_v):N_{K_w/F_v}^E(E(K_{w}))\right)$.
			\end{itemize}
		Here $\Delta(E)$ denotes the discriminant of $E$ over $F$, $\sigma_{v}$ denotes the real embedding of $F$ corresponding to $v \in M_F^{\infty}$, and $N_{K_w/F_v}^E$ denotes the local norm map.

		\end{itemize}
	\end{proposition}

	\begin{corollary} \label{corollary, if E(F)=N(E(K)) then Sha(ED/F) divides Sha(E/K)}
Keep the same assumptions and notations as in Proposition \ref{proposition, Qiu results}. Then the sequence \eqref{equation, EC-BRZI} can be rewritten as follows
			{\footnotesize
				\begin{equation} \label{equation, EC-BRZ in terms of S0 and infiinte places}
					0 \rightarrow    \Ker\left(\widetilde{\Res}_{K/F}^E\right) \rightarrow  H^1\left(G,E(K)\right)   \rightarrow \bigoplus_{v \in S_0 \cup M_F^{\infty}} H^1\left(G_{w},E(K_w)\right) \rightarrow \frac{\Ker(\trans^E_{K/F})}{\widetilde{\Res}_{K/F}^E\left(\Sha(E/F)\right)} \rightarrow  0.
			\end{equation}}
		Furthermore, if the above sequence is exact, i.e., the map $\mathcal{I}$ \eqref{equation, map I in Yu paper} is zero, then
			\begin{equation} \label{equation, order transE,K/F based on Qiu results}
		\# \Ker(\trans^E_{K/F})=\frac{2^{r_F+\delta(E,F,K)-r_{D,F}}.\#\Sha(E/F) }{\left(E(F):N^E_{K/F}(E(K)) \right)}.
			\end{equation}
	\end{corollary}
	
	\begin{proof}
		The first assertion immediately follows from Proposition \ref{proposition, Qiu results}. Also the equality \eqref{equation, order transE,K/F based on Qiu results} is obtained from Theorem \ref{theorem, order of ker(trans-tilde)} and the relation \eqref{equation, Refined Yu's formula by Qiu}.
	\end{proof}

	Using Proposition \ref{proposition, Qiu results}, we can find an infinite family of elliptic curves $E$ over quadratic extensions $K/F$ whose coarse Ostrowski quotients are trivial.
	
	
	\begin{example}
		Let  $K=\mathbb{Q}(\sqrt{p})$ where $p \equiv 1 (\mathrm{mod}\, 8)$ is an odd prime number, and $E$ be the $\mathbb{Q}$-curve  $y^2=x^3-x$. Denote by $E_p$ the $p$-twist of $E$, i.e., $E_p/\mathbb{Q}:y^2=x^3-p^2x$. One can show that the $\mathbb{Z}$-ranks of $E(\mathbb{Q})$ and $E_p(\mathbb{Q})$ are zero, and 
		\begin{equation*}
\left(E(\mathbb{Q}):N_{K/\mathbb{Q}}^{E}(E(K))\right)=4,
		\end{equation*}
	see \cite[proof of Theorem 3.3]{Qiu}. Hence by the part (i) of Proposition \ref{proposition, Qiu results} we have
	\begin{equation*}
	\# H^1(G,E(K))=4,
	\end{equation*}
where $G=\Gal(K/\mathbb{Q})$. On the one hand $\Sha(E/\mathbb{Q})=0$ \cite[proof of Theorem 3.3]{Qiu}, which implies that  $\widetilde{\Res}_{K/\mathbb{Q}}^E$, as in \eqref{equation, res-tilda}, is the zero map. By Lemma \ref{lemma, isomorphism between Ker(epsilon-tilda) and Ker(F)} we find
\begin{equation} \label{equation, H1 is embedded into sumH0s}
	H^1(G,E(K)) \hookrightarrow \bigoplus_{v \in M_{\mathbb{Q}}} H^1(G_w,E(K_w)).
\end{equation} 
On the other hand, since $\delta(E,\mathbb{Q},K)=2$ \cite[Lemma 3.2]{Qiu}, assuming the Tate-Shafarevich groups $\Sha(E/\mathbb{Q}(\sqrt{p}))$ and  $\Sha(E_p/\mathbb{Q})$ are finite, we can use the equality \eqref{equation, order of local H0's} to obtain
\begin{equation} \label{equation, product of order H0s are 4}
	 \prod_{v \in M_{\mathbb{Q}}} \# H^1(G_w,E(K_w))=4.
\end{equation}
Using the relations \eqref{equation, H1 is embedded into sumH0s} and \eqref{equation, product of order H0s are 4}, we conclude that $\Ost_c(E,K/F)=0$. Similarly one can show that $\Ost_c(E_p,K/F)=0$, see \cite[Remark (1) after Theorem 3.3]{Qiu}.
	
	\end{example}

\begin{remark}
 Recall that the notion of  Ostrowski quotient $\Ost(K/F)$ is defined for $K/F$ an \textit{arbitrary (not necessarily Galois)} finite extension of number fields, see Definition \ref{definition, Ostrowski group}. Whereas, in this paper, we investigated analogous notion of Ostrowski quotient for elliptic curves only over finite \textit{Galois extensions} of number fields. So it remains open to study elliptic analog of Ostrowski quotient in the non-Galois case. As a point of departure, regarding Remark \ref{remark, Ostrowski and coarse Ostrowski may coincide}, one may consider those elliptic curves $E$ over non-Galois extensions $K/F$ for which $E/K$ has good reductions at all primes of $K$. For instance, one may investigate (coarse) Ostrowski quotient of the $\mathbb{Q}$-curve 
	\begin{equation*}
		E: y^2=x^3+x^2-114x-127
	\end{equation*}
	over the cubic extension $\mathbb{Q}(\sqrt[3]{28})/\mathbb{Q}$, since $E/K$ has good reduction everywhere \cite[$\S$3, Example 1]{K}. Note that this curve is labeled 196B1 in \cite[p. 111]{Cre}, for which $E(F)$ is finite of order $\#E(F)=\#E(\mathbb{Q})=[K:\mathbb{Q}]=3$.
\end{remark}


	\bibliographystyle{amsplain}

\end{document}